\documentclass[11pt]{article}
\usepackage[T1]{fontenc}
\usepackage[utf8]{inputenc}
\usepackage{amsmath,amssymb,euscript}
\usepackage{bbm}
\usepackage{amsthm}
\usepackage[english]{babel}
\theoremstyle{plain}
\newtheorem{theorem}{Theorem}
\newtheorem{lemme}{Lemma}
\newtheorem{prop}{Proposition}
\newtheorem{rema}{Remark}

\usepackage{amsfonts}
\usepackage{amssymb}
\usepackage{makeidx}
\usepackage{graphicx}
\usepackage{stmaryrd}
\usepackage{caption,subcaption}
\usepackage{color}

\newcommand{\bbD}{{\ensuremath{\mathbb D}} }
\newcommand{\bbE}{{\ensuremath{\mathbb E}} }

\newcommand{\bbN}{{\ensuremath{\mathbb N}} }

\newcommand{\bbP}{{\ensuremath{\mathbb P}} }

\newcommand{\bbR}{{\ensuremath{\mathbb R}} }

\newcommand{\cL}{\ensuremath{\mathcal L}}
\newcommand{\cM}{\ensuremath{\mathcal M}}
\newcommand{\cN}{\ensuremath{\mathcal N}}

\newcommand{\cP}{\ensuremath{\mathcal P}}

\newcommand{\cR}{\ensuremath{\mathcal R}}

\newcommand{\gga}{\gamma}            % \gg already exists...

\newcommand{\1}[1]{\bf{1}_{#1}}							% indicatrice
\newcommand{\va}[1]{\vert #1 \vert}							% valeur abs
\newcommand{\p}[1]{\parallel #1 \parallel}					% norme

\newcommand{\E}[1]{\bbE \big[#1 \big]}						% espérance
\def\w{{\wedge}}
\newcommand{\Rp}[1]{\bbR_+^{#1}}                            % R+

\newcommand{\DT}[1]{\mathbb{D}([0,T], #1)}					% espace de Skorohod sur 0,T
\def\R{{\mathbb{R}^+}}

\def\me{\medskip \noindent}
\def\bi{\bigskip \noindent}

\def\be{\begin{eqnarray}}
\def\ee{\end{eqnarray}}
\def\ben{\begin{eqnarray*}}
\def\een{\end{eqnarray*}}

\title{Large fluctuations in multi-scale modeling for rest  erythropoiesis}
\author{C\'eline Bonnet\thanks{CMAP, Ecole Polytechnique, CNRS, IP Paris,  route de
    Saclay, 91128 Palaiseau Cedex-France; E-mail: \texttt{celine.bonnet@polytechnique.edu}}, Sylvie M\'el\'eard\thanks{CMAP, Ecole Polytechnique, CNRS, IP Paris, IUF, route de
    Saclay, 91128 Palaiseau Cedex-France; E-mail: \texttt{sylvie.meleard@polytechnique.edu}}}
    
    \date{\today}
\begin{document}
	\maketitle
	\begin{abstract}
		Erythropoiesis is a mechanism for the production of red blood cells by cellular differentiation. It is based on amplification steps  due to an interplay between renewal and differentiation in the successive cell compartments  from stem cells to red blood cells.  We will study this mechanism with a stochastic point of view to explain unexpected fluctuations on the red blood cell numbers, as surprisingly observed by biologists and medical doctors in a rest erythropoiesis.
		We consider  three compartments:  stem cells, progenitors and red blood cells. The dynamics of each cell type is characterized by its division rate and by the renewal and differentiation probabilities at each division event. We model the global population dynamics by a three-dimensional stochastic decomposable  branching  process. We show that the amplification mechanism is given by the inverse of the small difference between the differentiation and renewal probabilities. Introducing a  parameter $K$ which scales simultaneously the size of the first component, the differentiation and renewal probabilities and the red blood cell death rate, we describe the asymptotic behavior of the process for large  $K$. 
		We show that each compartment has its own size scale and its own time scale. Focussing on the third component, we prove that  the red blood cell population size, conveniently renormalized (in time and size), is expanded in an usual way  inducing large fluctuations. The proofs are based on a fine study of the different scales involved in the model 
		and on the use of different convergence and average techniques in the proofs.

\end{abstract}

\textbf{Keywords} Decomposable branching process; Multi-scale approximation; Stochastic slow-fast dynamical system; Large fluctuations; Rest Erythropoiesis; Amplification mechanism.

	\section{Introduction}

 The model and the stochastic behavior we are studying in this paper are based on the biological mechanisms of rest (without stress) erythropoiesis. Erythropoiesis is a mechanism for the production of red blood cells by cellular differentiation of stem cells. Stem cells, although in large numbers, produce even more red blood cells per day using a specific amplification mechanism.
%  Cell differentiation involves in fact several types (from stem cells to red blood cells), and each compartment (defined by a cell type) has an amplified size compared to the previous one. 
 
 We will study this amplification mechanism using a decomposable branching process (see \cite{branching2010}-Section 12, \cite{kimmelbranching}-Section 6.9.1, \cite{till-culloch64}). Such process allows in particular to capture the genealogy of the cells, including the history of their types. 
 
 Let us firstly describe more precisely the biological dynamics, then we will introduce the mathematical model. The dynamics of erythropoietic cells, at rest, results in two distinct events, renewal and differentiation. Indeed, each cell of each type (except the last one) divides into two cells at a constant rate, depending on its type. These two new cells are either of the same type as the mother cell (renewal) or of the "next" cell type (differentiation). The  final stage of differentiation corresponds to  red blood cells which don't divide and can only die at a constant rate.  
The stem cells    are those with the highest capacity for renewal, but not so high to prevent the cell population to explode. Further, the amplification from one compartment (characterized by one type) to the next one is proportional to the inverse of the difference between the differentiation and renewal probabilities, which  is small. Note also that the death rate in the last compartment plays a main role. 

\me We are interested in describing  the stochastic fluctuations of the compartment sizes for the rest erythropoiesis. In this case, the regulation doesn't play any role but nevertheless one observes unusual large oscillations at the red blood cell level. Indeed, the red blood cells number, in a human rest erythropoiesis, varies by $10\%$ around its average value (cf.  \cite{thirup2003}). The order of magnitude of these variations is greater than the one of the classical variations for multi-type branching processes, which  should be of the order  $0.001\%$.

\bi In this paper, we will model the differentiation steps by considering $3$ types. These types correspond to stem cells (type $1$), progenitors (type $2$) with the ability in amplifying the cells number, and red blood cells (type $3$). The number of stem cells in the initial state will be characterized by a (large) scaling parameter $K\in \mathbb{N}^*$. 

\bi  Let us now introduce more precisely the parameters of the dynamics.

\bi  Cells of type 1 evolve according  a critical linear birth and death process. Birth events correspond to renewal division events, occurring at rate $\frac{\tau_1}{2}>0$, while death events correspond to differentiation events occurring at the same rate (a cell of type 1 divides in two cells of   type 2).  Cells of  type 2 divide at rate $\tau_2>0$ in two cells of the same type (renewal event) with probability $p_2^R$ and in two cells of type 3 (differentiation event) with probability $p_2^D = 1- p_2^R \in ]1/2,1[$. The cells of type 3 are mature cells which die at rate $d_3>0$. We can summarize the dynamics as follows. If $(N_{1}, N_{2}, N_{3})$ denotes the vector of sub-population sizes, the transitions of the hematopoietic process are given by
\begin{align*}
	N_{1}  \quad\longrightarrow & \quad N_{1}+1&& \text{at rate } (\tau_{1}/ 2)\,N_{1}\\
	(N_{1},N_{2})  \quad\longrightarrow & \quad(N_{1}-1,N_{2}+2) && \text{at rate } (\tau_{1}/ 2)\,N_{1}\\
	N_{2}  \quad\longrightarrow & \quad N_{2}+1&& \text{at rate } \tau_{2}\, p^R_{2} \,N_{2}\\
	(N_{2},N_{3}) \quad \longrightarrow &\quad (N_{2}-1,N_{3}+2)&& \text{at rate } \tau_{2}\, p^D_{2}\,N_{2}\\
	N_{3}  \quad\longrightarrow & \quad N_{3}-1&& \text{at rate } d_{3}\,N_{3}.
\end{align*}

%\begin{eqnarray*}
%	N_{1}& \longrightarrow &N_{1}+1  \quad \hbox{ at rate } (\tau_{1}/ 2)\,N_{1}\\
%	N_{1}& \longrightarrow &N_{1}-1  \quad \hbox{ at rate } (\tau_{1}/ 2)\,N_{1}\\
%	N_{2}& \longrightarrow &N_{2}+2  \quad \hbox{ at rate } (\tau_{1}/ 2)\,N_{1}\\
%	N_{2}& \longrightarrow &N_{2}+1  \quad \hbox{ at rate } \tau_{2}\, p^R_{2} \,N_{2}\\
%	N_{2}& \longrightarrow &N_{2}-1  \quad \hbox{ at rate } \tau_{2}\, p^D_{2}\,N_{2}\\
%	N_{3}& \longrightarrow &N_{3}+2  \quad \hbox{ at rate } \tau_{2}\, p^D_{2} \,N_{2}\\
%	N_{3}& \longrightarrow &N_{3}-1  \quad \hbox{ at rate } d_{3}\,N_{3}.
%\end{eqnarray*}

%\me Let us remark that differentiation events act as death events on the population process.   

\me Here, we have assumed that each division is symmetric, so that 
	\be
	\label{sym} p_2^D+p_2^R = 1.\ee
	We could have included asymmetric division without  changing the results of our study. Indeed it doesn't change the main characteristics of the dynamics.

\bi As explained above, the number of cells of each type is large, but moreover, there is an amplification mechanism between the compartments, based on the  small difference $ p_2^D-p_2^R$ between the differentiation and renewal probabilities in compartment 2, and on the small death rate $d_{3}$, in a way which is now defined.

\me   We assume that\\
$\bullet$ the size of the type $1$-cells population is of order $K$,\\
$\bullet$  there exists  a couple of positive parameters $(\gamma_2,\gamma_3) \in ]0,1[$ such that
 \be
\label{proba} p_2^D-p_2^R = K^{-\gamma_2} \quad \text{and } \quad d_3 = \tau_3 K^{-\gga_3} \quad \text{ with } \tau_3>0.
\ee
Let us note that \eqref{sym} and \eqref{proba} make  the probabilities $p^R_2$ and $p^D_2$ depend on $K$, 
$$p^D_2 = 1-p^R_2=1/2 + K^{-\gamma_2}/2.$$
Therefore the dynamics in this compartment is close to a critical process.

\me Assumptions \eqref{proba} introduce the  different time and size scales playing  a role for  the multi-scale population process  describing the dynamics of each compartment size. Hence we will denote by $N^K$, the population process N previously defined.

\me We assume in the following that 
\be
\label{gamma}\gga_2 <\gga_3<1.
\ee

%\me This case is close to the biological observations. At the same time, the case $\gga_3 <\gga_2 <1$ doesn't bring any new difficulty and can be solved using the same arguments as those detailed in the article.

\me This case is the most interesting mathematically and closest to the biological observations. Indeed, in a more realistic model with a larger number of compartments based on biological observations (see Bonnet at al \cite{notrepapier}), we observe that the red blood cell death rate drives the slowest time scale.

\bi  	Our aim in this paper is to finely describe this dynamics, when $K$ goes to infinity, using appropriate renormalizations. 
	
	  We will see that a size renormalization is not enough to describe the dynamics of the last two components of the process. A time renormalization is also necessary. More precisely each compartment has its own size scale, of order $K$ for Compartment 1,  $K^{1+\gamma_{2}}$ for Compartment $2$ (resp. $K^{1+\gamma_{2}+ \gamma_{3}}$ for Compartment $3$) and its own time scale, of order 1 for Compartment 1,  $K^{\gamma_{2}}$ for Compartment $2$ (resp. $K^{\gamma_{3}}$ for Compartment $3$). 	  
	
	   \me The next simulations show the dynamics of the process in the typical time scale of each compartment, namely $K$, $K^{\gamma_2}$ and $K^{\gamma_3}$. 
 We take as initial condition $$N^K(0) = (K,0,0)$$
and choose  $\  K = 2000, \quad \gamma_2 = 0.55, \quad \gamma_3 = 0.8.$ Hence $K^{\gga_2} \sim 60$ and $K^{\gga_3} \sim 400$.

\noindent The others parameters are equal to 1.  

Figure \ref{fig:Nt} shows the simulation of a trajectory of the process $(N^K(t), \quad t \in [0,T])$ for $T \sim 1$, decomposed on the three compartments. Figure \ref{fig:oscillation} shows the simulation of a trajectory of the process $(N^K(t), \quad t \in [0,T])$ for $T \sim K^{\gga_2}$ and Figure \ref{fig:NtK} shows the simulation of a trajectory of the process $(N^K(t), \quad t \in [0,T])$ for $T \sim K^{\gga_3}$. 
The  horizontal orange line gives the order of magnitude for each compartment size ($K$, resp. $K^{1+\gamma_{2}}$, $K^{1+\gamma_{2}+\gamma_{3}}$).

\begin{figure}[!h]
	\centering
	\begin{subfigure}[b]{0.3\textwidth} % "0.45" donne ici la largeur de l'image
		\centering \includegraphics[width=\textwidth]{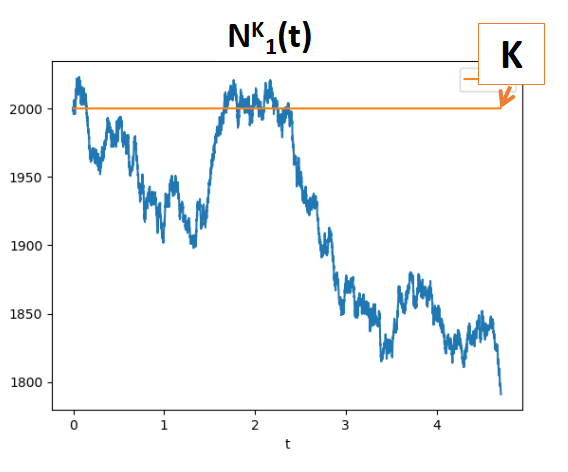}
		%	 		\caption{Orchid}\label{fig:N1t}
	\end{subfigure}
	~ % ce symbole ajoute un espacement horisontal entre les premiÃ¨res deux images
	\begin{subfigure}[b]{0.33\textwidth}
		\centering \includegraphics[width=\textwidth]{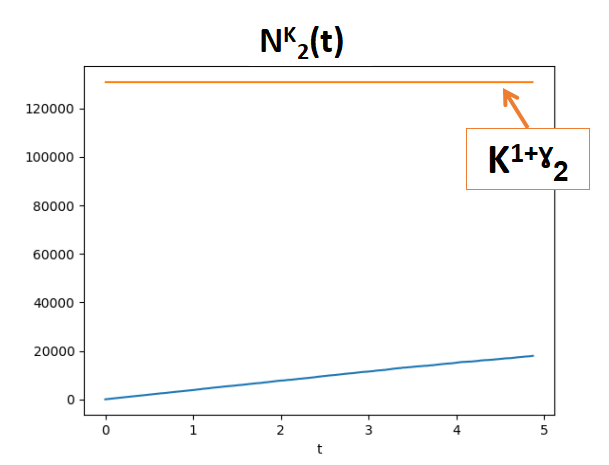}
		%	 		\caption{Calla Lily}\label{fig:N2t}
	\end{subfigure}
	~ 
	\begin{subfigure}[b]{0.31\textwidth}
		\centering \includegraphics[width=\textwidth]{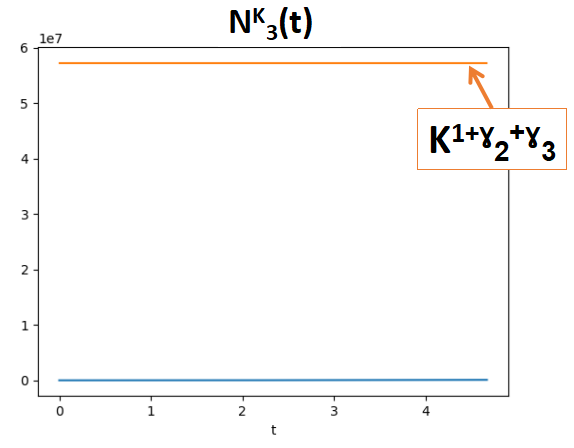}
		%	 		\caption{N3t}\label{fig:N3t}
	\end{subfigure}
	\caption{A trajectory of the $N^K$ process for $t \in[0,T]$ with $T= O(1)$}\label{fig:Nt}
\end{figure}
\begin{figure}[!h]
	\centering
	\begin{subfigure}[b]{0.3\textwidth} % "0.45" donne ici la largeur de l'image
		\centering \includegraphics[width=\textwidth]{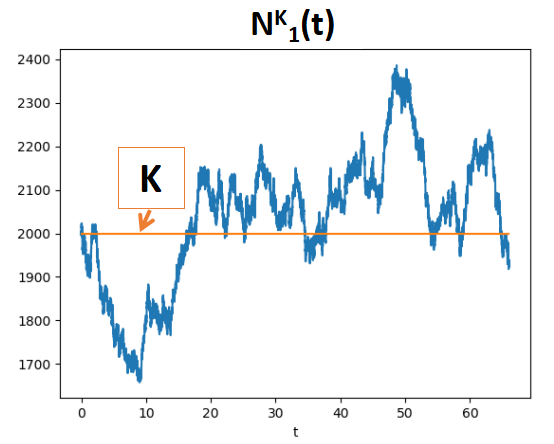}
		%	 		\caption{Orchid}\label{fig:N1t2}
	\end{subfigure}
	~ % ce symbole ajoute un espacement horisontal entre les premiÃ¨res deux images
	\begin{subfigure}[b]{0.3\textwidth}
		\centering \includegraphics[width=\textwidth]{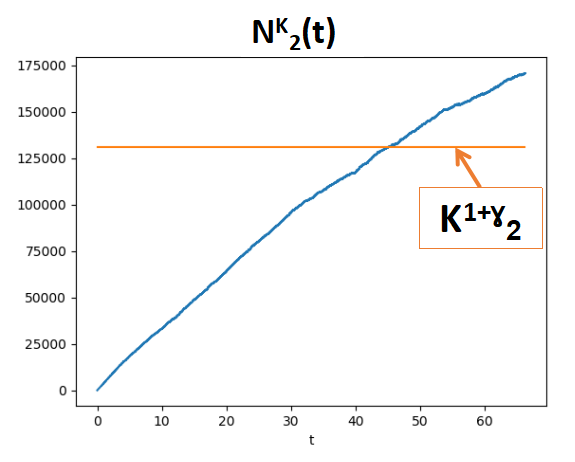}
		%	 		\caption{Calla Lily}\label{fig:N2t2}
	\end{subfigure}
	~ 
	\begin{subfigure}[b]{0.31\textwidth}
		\centering \includegraphics[width=\textwidth]{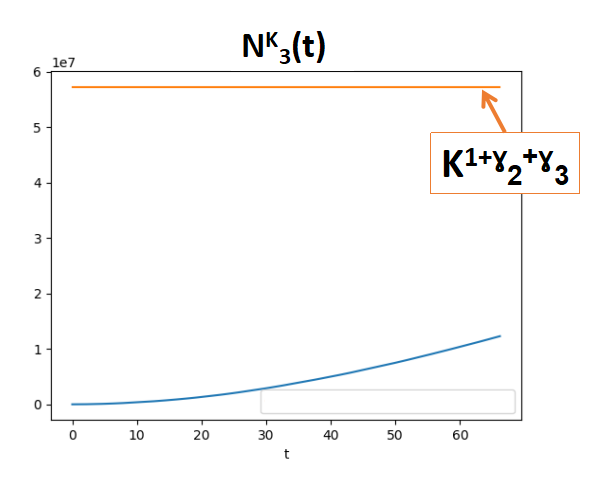}
		%	 		\caption{N3t}\label{fig:N3t2}
	\end{subfigure}
	\caption{A trajectory of the $N^K$ process for $t \in[0,T]$ with $T= O(K^{\gga_2})$}\label{fig:oscillation}
\end{figure}
\begin{figure}[!h]
	\centering
	\begin{subfigure}[b]{0.32\textwidth} % "0.45" donne ici la largeur de l'image
		\centering \includegraphics[width=\textwidth]{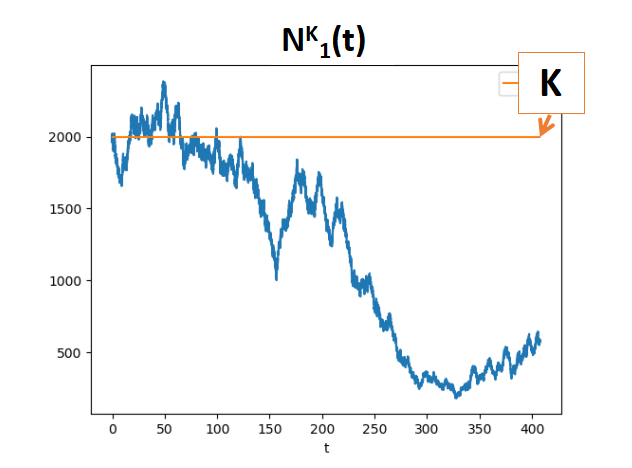}
		%	 		\caption{Orchid}\label{fig:orchid}
	\end{subfigure}
	~ % ce symbole ajoute un espacement horisontal entre les premiÃ¨res deux images
	\begin{subfigure}[b]{0.31\textwidth}
		\centering \includegraphics[width=\textwidth]{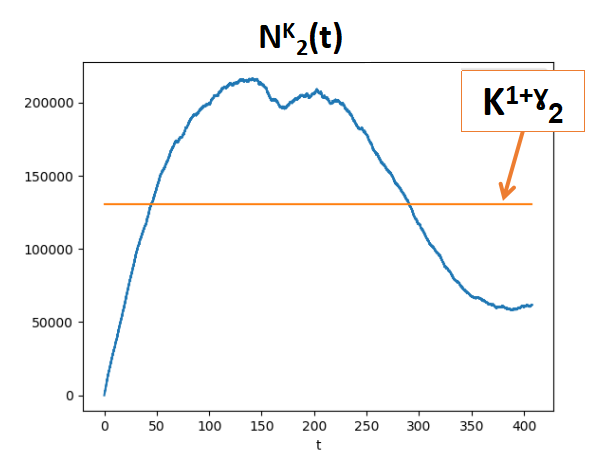}
		%	 		\caption{Calla Lily}\label{fig:calla}
	\end{subfigure}
	~ 
	\begin{subfigure}[b]{0.3\textwidth}
		\centering \includegraphics[width=\textwidth]{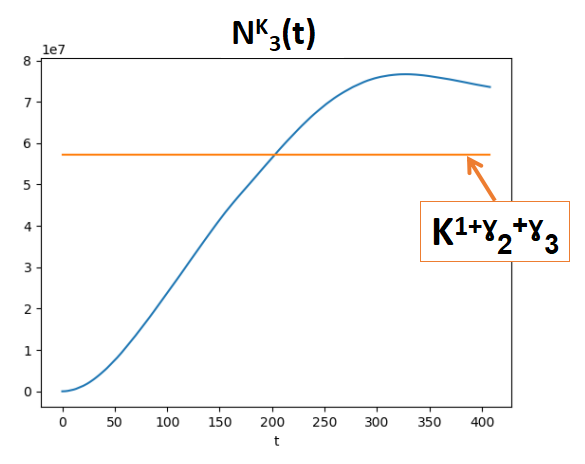}
		%	 		\caption{Rose}\label{fig:rose}
	\end{subfigure}
	\caption{A trajectory of the $N^K$ process for $t \in[0,T]$ with $T= O(K^{\gga_3})$}\label{fig:NtK}
\end{figure}

\bi We observe that at a time scale of order $1$, the  two last components of the process $N^K$ are far from their equilibrium size.
We observe in Figure 2 that the two first components of $(N^K(t), \quad t \in [0,T])$ for $T \sim K^{\gga_2}$ evolve around their equilibrium size, which is not the case of the third one. 	
In Figure 3, the process is considered on a longer period of time, $T \sim K^{\gga_3}$ and one sees that the third component hits a neighborhood of its equilibrium.
 Furthermore, we observe the oscillations of the components of $N^K$ around their equilibrium. We note that they are smoother and smoother from Compartment 1 to Compartment 3 and that the  amplitude of the waves is longer and longer.

\me Compartment 3 illustrates the particular behavior of the red blood cells in a rest human erythropoiesis, highlighted above. Indeed, the expected variations  should be around $K^{-(1+\gga_2+\gga_3)/2} \sim 0.01\%$, and we observe variations which seem to be of order  $\frac{10^7}{K^{(1+\gga_2+\gga_3)/2}} \sim 17\%$.

%Their number varies by $10\%$ around their average value (cf.  \cite{thirup2003}). The order of magnitude of these variations is much higher than that of the classical variations observed for multi-type branching processes in large populations. 

\me Our aim is to prove and quantify these different behaviors and to explain these large fluctuations.  

\me  In Section 2, basic martingale properties are stated, some estimates are given for the moments of compartments sizes and a first study on the convergence of the process at a time-scale of order $1$, when $K$ tends to infinity, is given. We show that the two last components do not reach their equilibrium at this time scale. In Section 3, we study the process on an appropriate time-scale to capture the asymptotic behavior of the second and third types. We show that the limiting behavior of the two first components process at the time scale of order $K^{\gamma_{2}}$ is given by an explicit deterministic function $y$. We also show that this time scale is not long enough to observe the dynamics of the third component. Hence, we study the limiting behavior of the process at the time scale of order $K^{\gamma_{3}}$. At this time scale, the second component process goes too fast and doesn't converge anymore. We consider the associated occupation measure, as already done in Kurtz \cite{kurtz92}. We prove its convergence in a weak sense to the Dirac measure at the unique equilibrium of the second component of the deterministic function $y$. Then we deduce the convergence of the third component to a limiting deterministic system involving this equilibrium. In Section 4, we study and quantify the large fluctuations observed in the simulations. We show that the first component  behaves at  the different time scales as a Brownian motion. This is not the case for the other two components. Theorem \ref{thU} describes the second and third order asymptotics of the second component on its typical time and size scale. The fluctuations around its deterministic limit are not Gaussian. They are described by a finite variation process integrating the randomness of the first component. An independent Brownian motion appears in the third order term. Theorem \ref{thV}, which is the main theorem of the paper, describes the fluctuations associated with the third component dynamics. We show that the randomness induced by the dynamics of the two last components is negligible. To capture the effect of the randomness of the first component imposes a size-scale which allows to observe the large oscillations of the third component. We identify these oscillations as a finite variation process integrating as above the fluctuations of the first component.   

\bi

The mathematical modeling of hematopoiesis has been firstly introduced in the seminal paper of Till, McCulloch and Siminovitch \cite{till-culloch64}. In this paper, the authors study a binary branching process and show, comparing with biological results, that the probabilistic framework is  relevant. Since this pioneering work, many mathematical approaches have been proposed to describe more precisely the cell differentiation kinetics, based either on deterministic or stochastic models (a survey concerning many models can be found in \cite{kimmel-whichard2010}). A deterministic approach consists in introducing a dynamical system describing the behavior of the different compartments and in studying different properties of this system, in particular the equilibrium states (see for example \cite{crauste08}, \cite{loe80}, \cite{marciniak09}, \cite{arinokimmel86} and the references therein). One can also add  a noise to model some random perturbation of these systems, with an eventual delay (see for example \cite{lei07}, \cite{Pa}).  In  \cite{doumic2010},  a continuous description of the different cells types is also proposed, using a partial differential equation. Moreover, in all these papers, the authors are interested in modeling the regulation which happens when the system is perturbed by some stress and this nonlinearity involves many mathematical difficulties. 
 Let us note that  other stochastic models for hematopoiesis have been introduced  (\cite{abk00}, \cite{roe02},  \cite{kimmel82}) but they concentrate on a specific level (either stem cells or red blood cells). In all this literature, the questions studied by the authors don't concern the impact of  the  parameters on the amplification mechanism.  We have found only two papers,
 \cite{Din07} and \cite{marciniak09}, in which the question is mentioned. To our knowledge, the fluctuations generated by this amplification mechanism, have never been rigorously studied with such a space-time multiscale point of view.

Most of the slow fast dynamical systems model interaction between species with different  behaviors driving the time scales (see for example \cite{pelevs2006reduction}). In such cases, slow and fast components appear naturally, contrary to our case, for which a fine study is needed  to find the specific time scale of each compartment.
In the other way, Popovic, Kurtz and Kang in \cite{lea2014} have developed a general theorem for convergence and fluctuations of multiscale processes. Their result can't explain our asymptotics. Indeed, in their result, the fluctuations around the deterministic behavior of the slow component are Gaussian, which is not the case of the red blood cells dynamics previously described. In their work, the martingale part of the limit is due to two sources of randomness: the slow component dynamics and the averaged effects of the fast components  on the slow component dynamics. As  previously explained, in our case the randomness of the slow component is only due to the fast ones and its intrinsic randomness is negligible.

\bi \textit{Notation.}
 $\cP(E)$ and $\cL(X)$ will denote respectively the space of probability measures on $E$ and the law of a process $X$.
As in \cite{kurtz92}, we will denote by $l_m(\Rp{})$ the space of measures on $[0,\infty) \times \Rp{} $ such that $\mu([0,t]\times \Rp{})=t $, for each $t \geq 0$. \\

\section{Amplification mechanism : size-scale dynamics}
\subsection{The amplification mechanism}

As explained in the introduction, we will reduce the model by simplicity  to three compartments. The first one will describe the stem cells compartment, the second one will describe the compartment of progenitors and the third one will refer to red blood cells. Also by simplicity, we will describe the type of cells in each compartment by type 1, type 2 and  type 3.

\bi Let us now introduce  the vector $N^K(t)= (N_1^K(t), N_2^K(t), N_3^K(t))$ of population sizes  at time $t$. The  process $N^K$ is a decomposable multi-type branching process, that is a Markov jump  process  whose  dynamics is given by the following equations. 

\bi  We assume that for any fixed $K$, $N^K_{1}(0), N^K_{2}(0), N^K_{3}(0)$ are  integrable.

\me Let us denote by $(\cN_i^{j})_{\underset{j\in\{-,+\}}{1\leq i\leq 3}}$ independent Poisson point measures with intensity $dsdu$ on $\Rp{2}$ and introduce the filtration $(\mathcal{F}_t)_{t \geq 0}$ given by $$\mathcal{F}_t = \sigma(\cN_i^{j}([0,s) \times A); \, i \in \{1,\dots, 3\},j\in\{-,+\}, \,s \leq t, \, A \in \mathcal{B}(\R) ).$$ 

Then we have
\begin{equation} \label{IDS}
\begin{split}
N^K_1(t) =\, N_1^K(0)& + \int_{0}^{t}\int_{\Rp{}} \1{u \leq \frac{\tau_1}{2}\,N_1^K(s^-)} \,\, \cN^+_1(ds, du) - \int_{0}^{t}\int_{\Rp{}} \1{u \leq \frac{\tau_1}{2}\,N_1^K(s^-)} \,\, \cN^-_1(ds, du) \\
N^K_2(t) =\,  N^K_2(0)& + 2 \, \int_{0}^{t}\int_{\Rp{}} \1{u \leq \frac{\tau_1}{2} \,N^K_1(s^-)} \,\, \cN^-_1(ds, du)  + \int_{0}^{t}\int_{\Rp{}} \1{u \leq \tau_2 p_2^R \, N^K_2(s^-)} \,\, \cN^+_2(ds, du) \\
&- \int_{0}^{t}\int_{\Rp{}} \1{u \leq \tau_2 p_2^D \,N^K_2(s^-)} \,\, \cN^-_2(ds, du) \\
N^K_3(t) =\, N^K_3(0)& + 2 \, \int_{0}^{t}\int_{\Rp{}} \1{u \leq \tau_2 p_2^D \,N^K_2(s^-)} \,\, \cN^-_2(ds, du)  - \int_{0}^{t}\int_{\Rp{}} \1{u \leq \tau_3 K^{-\gamma_3} \,N^K_3(s^-)} \,\, \cN^-_3(ds, du) 
\end{split}.
\end{equation}

It can be written as
\begin{equation} \label{semi}
\begin{split}
\forall t\geq 0, \quad N^K_1(t) &=\, N^K_1(0) + M^K_1(t) \\
N^K_2(t) &=\, N^K_2(0) + \tau_1\,\int_{0}^{t}  N^K_1(s) \, ds - \tau_2 \,K^{-\gamma_2} \, \int_{0}^{t}  N^K_2(s) \, ds + M^K_2(t) \\
N^K_3(t) &=\, N^K_3(0) + 2\, \tau_2 \,p^D_2 \,\int_{0}^{t} N^K_2(s) \, ds - \tau_3\, K^{-\gamma_3} \,\int_{0}^{t}  N^K_3(s) \,ds + M^K_3(t)
\end{split}
\end{equation}
where $M^K = (M^K_1,M^K_2,M^K_3)$ is a square-integrable martingale such that for all $t\geq0$,
\begin{equation} \label{crochet}
\begin{split}
<M^K_1>_t &=\, \tau_1\,\int_{0}^{t}  N^K_1(s) \,ds \\
<M^K_2>_t &=\,2\, \tau_1\, \int_{0}^{t}  N^K_1(s) \, ds + \tau_2 \, \int_{0}^{t}  N^K_{2}(s) \, ds \\
<M^K_3>_t &=\,4\, p^D_2 \tau_2 \, \int_{0}^{t}  N^K_2(s) \, ds +  \tau_3 K^{-\gga_3} \,  \int_{0}^{t} N^K_3(s) \,ds \\
<M^K_1,M^K_2>_t &=\, -\,\tau_1\, \int_{0}^{t}  N^K_1(s) \, ds \\
<M^K_2,M^K_3>_t &=\, -2\,p^D_2 \tau_2 \, \int_{0}^{t}  N^K_2(s) \, ds.
\end{split}
\end{equation}

Indeed, by standard localization and Gronwall's arguments applied to $(N^K_1(t))_{t}$,  we can easily prove that  for any $T>0$ and $K\in \mathbb{N}^*$,
\be
\label{controlN1}
\E{\sup_{t\leq T} N^K_{1}(t)} \leq (2 +\E{N^K_{1}(0)} )e^{2\tau_{1} T},
\ee
and then that \be
\E{\sup_{t\leq T} N^K_{2}(t)} <+\infty\ ; \ \E{\sup_{t\leq T} N^K_{3}(t)} <+\infty.
\ee

\bi
	  We obtain  from \eqref{semi} that the function $t \mapsto n(t) = \E{N^K(t)}=(n_1(t),n_2(t),n_3(t))$ satisfies the system of equations
	\begin{align} \label{expectation}
	 \forall t\leq T, \begin{cases}
	\quad n_1(t) &=\, \E{N^K_1(0)} \\
	\frac{d}{dt}n_2(t) &=\,  \tau_1\, n_1(t) -  \tau_2 K^{-\gga_2} \, n_2(t)\\
	\frac{d}{dt} n_3(t) &=\, 2 \tau_2 p_2^D \, n_2(t) - \tau_3 K^{-\gga_3} \, n_3(t).\end{cases}
	\end{align}

\bi	As explained in the introduction, \be
\label{n1}
\E{N_1^K(0)} \sim K.
\ee

Therefore there is a unique equilibrium given by
	\begin{align}\label{magnitude1}
	\forall t\geq 0, \quad n_1^* &=\, \E{ N^K_1(0) } \sim K\nonumber\\
	n_2^* &=\,  \frac{\tau_1\,n_1^* }{ \tau_2} K^{\gga_2} \sim K^{1+\gga_2}\nonumber\\
	n_3^* &=\, \frac{2 p^D_2 \tau_2  \, n_2^* }{ \tau_3 }K^{\gga_3}\sim K^{1+\gga_2+\gga_3}
	\end{align}

\begin{rema}
	In the above  computation, we obtain the order of magnitude of each sub-population size at equilibrium, but we cannot deduce the order of magnitude of the time taken by the process to reach this equilibrium.
	We will keep this remark in mind in all the paper.
\end{rema}

\bi
Let us first begin by a lemma showing that for any $K$ and $t$, the expectations of the sub-population sizes behave as expected from  \eqref{magnitude1}.

\begin{lemme} \label{preliminary}
	Let us now assume that  $$ \displaystyle \sup_{K} \E{\frac{N^K_1(0)}{K}} < +\infty, \quad \sup_{K} \E{\frac{N^K_2(0)}{K^{1+\gga_2}}} <+ \infty, \quad \sup_{K} \E{\frac{N^K_3(0)}{K^{1+\gga_2+\gga_3}}} <+ \infty.$$ 
	then $$\displaystyle \sup_{K,\, t \in \Rp{}} \E{\frac{N^K_1(t)}{K}} < + \infty\ , \displaystyle \sup_{K,\, t \in \Rp{}} \E{\frac{N^K_2(t)}{K^{1+\gga_2}}} < + \infty\ , \ \displaystyle \sup_{K,\, t \in \Rp{}} \E{\frac{N^K_3(t)}{K^{1+\gga_2+\gga_{3}}}} < +\infty.$$  
\end{lemme}

\begin{proof}
	
	The first assertion follows immediately.
	
	\bi 	From \eqref{expectation}, we obtain that for all $t \geq 0$, \be \label{EN2} \E{\frac{N^K_2(t)}{K^{1+\gga_2}} }= \frac{ \tau_1}{\tau_{2}}\, \E{\frac{N^K_1(0)}{K}} + \big( \E{\frac{N^K_2(0)}{K^{1+\gga_2}}} - \frac{ \tau_1}{\tau_2}  \E{\frac{N^K_1(0)}{K}}\big) \, e^{-\tau_2\,K^{-\gga_2} \, t} \ee
	and the proof of the second assertion  follows.
	
	\me Similarly, straightforward computation yields
	\begin{align} \label{EN3}
	 \E{\frac{N^K_3(t)}{K^{1+\gga_2+\gga_{3}}} } = \;& \frac{2 p^D_{2} \tau_1}{\tau_{3}}\,\E{\frac{N^K_1(0)}{K}} -  \beta_K \, e^{-\tau_2\,K^{-\gga_2} \, t} \\
	& +\Big( \E{\frac{N^K_3(0)}{K^{1+\gga_2+\gga_{3}}}} - \frac{2 p^D_{2} \tau_1}{\tau_{3}}\,\E{\frac{N^K_1(0)}{K}} +\beta_K \Big) \, e^{-\tau_3\,K^{-\gga_3} \, t},\nonumber
	\end{align}
	with $$\beta_K = {2 p^D_{2} \tau_2}\, \frac{1}{\tau_{2}K^{\gga_3 - \gga_2}-\tau_{3}} \Big( \E{\frac{N^K_2(0)}{K^{1+\gga_2}}} - \frac{ \tau_1}{\tau_2}  \E{\frac{N^K_1(0)}{K}}\Big).$$
	Hence the third assertion is proved. 
	
\end{proof}

\subsection{Asymptotic behavior on a finite time interval}

The parameter $K$ is defined as the order of magnitude of  the martingale $N^K_{1}$ at time $0$.

The  first result  describes the dynamics of the process on a finite time interval. The proof of this proposition is  left to the reader. It is classical and more difficult proofs in a similar spirit will be given later. 

\begin{prop} \label{thX}	 
	Let us introduce the jump process $X^K$ defined  for all $t \geq0$ by
	\be
	\label{X}
	X^K(t)=(\frac{N^K_1(t )}{K}, \frac{N^K_2(t)}{K^{1+\gga_2}}, \frac{N^K_3(t)}{K^{1+\gga_2+\gamma_{3}}}). 
	\ee
	(i)  Let us assume that there exists a vector $(x_1,0,0) \in \Rp{3}$ such that the sequence \\$\Big(\frac{N^K_1(0 )}{K}, \frac{N^K_2(0)}{K^{1+\gga_2}}, \frac{N^K_3(0)}{K^{1+\gga_2+\gamma_{3}}}\Big) _{K \in \bbN^*}$ converges in law to $(x_1,0,0)$ when $K$ tends to infinity and such  that 
	$$\displaystyle \sup_{K} \E{ \frac{N^K_1(0 )}{K}} < +\infty, \quad \sup_{K} \E{ \frac{N^K_2(0)}{K^{1+\gga_2}}} < +\infty \quad \text{ and } \quad \sup_{K} \E{ \frac{N^K_3(0)}{K^{1+\gga_2+\gamma_{3}}}} < +\infty.$$
	
	Then
	for all $T>0$, the sequence $(\frac{N^K_1(t )}{K}, \frac{N^K_2(t)}{K^{1+\gga_2}}, \frac{N^K_3(t)}{K^{1+\gga_2+\gamma_{3}}})_{K \in \bbN^*}$ converges in law in $\mathbb{D}([0,T],\mathbb{R}_{+}^3) $ to $(x_1, 0,0)$.
	
	\bi(ii)  Let us assume that there exists a vector $(x_1,0,0) \in \Rp{3}$ such that the sequence \\$\Big(\frac{N^K_1(0 )}{K},  \frac{N^K_2(0)}{K},\frac{N^K_3(0)}{K}\Big) _{K \in \bbN^*}$ converges in law to $(x_1 ,0,0)$ when $K$ tends to infinity and such  that 
	$$\displaystyle \sup_{K} \E{ \frac{N^K_1(0 )}{K}} < +\infty, \quad \sup_{K} \E{ \frac{N^K_2(0)}{K}} < +\infty \quad \text{ and } \quad \sup_{K} \E{ \frac{N^K_3(0)}{K}} < +\infty.$$
	
	Then
	for all $T>0$, the sequence $\Big(\frac{N^K_1(t )}{K},  \frac{N^K_2(t)}{K},\frac{N^K_3(t)}{K}\Big)_{K \in \bbN^*}$ converges in law in  $\mathbb{D}([0,T],\mathbb{R}_{+}^3) $ to $\ x_1\big(1, \tau_{1}\, t ,  {\tau_{2}\over 2}\;  
	t^2\big)$.
	
\end{prop}

\bi 
Let us underline that  at this time scale,  assertion (i) shows that the two last components do not reach their equilibrium order, as observed in the simulations. Assertion (ii) proves that the three compartments are only of order $K$ during the time interval $[0,T]$. 

%"""""""""""""""""""
%
%\bi One can also easily prove that 	the first component behaves as a Brownian motion at the second order: for large $K$, for all $t\in[0,T]$,
%$$N^K_1(t ) \sim x_1\,K + \sqrt{K}\,\sqrt{\tau_{1} x_1 } \, B_{t}	,$$
%which explains the simulations of Figure 1. 
%"""""""""""""""""""

\section{Size-time multi-scale dynamics and asymptotic behavior}
A size renormalization of the stochastic process is not enough to understand the dynamics of the model. We need to change the time scale as observed in the simulations. This part is devoted to identify and study two significant asymptotics for the process, corresponding to the two time scales $K^{\gamma_{2}}$ and $K^{\gamma_{3}}$, when $K$ is large. 

\subsection{Asymptotic behavior at a time-scale of order $K^{\gamma_{2}}$}

    In this section, we  study the system composed of the two first components at the time scale $K^{\gga_{2}}$. To this end, let us introduce the jump process $Y^K$ defined  for all $t \geq0$ by
\be
\label{Y}Y^K(t)=(\frac{N^K_1(t \, K^{\gga_2})}{K}, \frac{N^K_2(t\, K^{\gga_2})}{K^{1+\gga_2}}). \ee

%Note that at time $t=0$, we have $$(Y_1^K(0),Y_2^K(0))=(X_1^K(0),X_2^K(0)).$$
Let us note that only the time scale differs between processes $X^K$ and $Y^K$. Hence, at time $t=0$, we have $$(Y_1^K(0),Y_2^K(0))=(X_1^K(0),X_2^K(0)).$$

\bi The next theorem  describes the approximating behavior of $Y^K$ when $K$ tends to infinity.  
\begin{theorem} \label{thY}		
	
	Assume that there exists a vector $(x_1,x_2) \in \Rp{2}$ such that the sequence $(Y^K(0))_{K \in \bbN^*}$ converges in law to $(x_1,x_2)$ when $K$ tends to infinity and such that $$\displaystyle \sup_{K} \E{ Y^K_1(0)^2 +  Y^K_2(0)^2} < \infty.$$
	Then for each $T>0$, the sequence $(Y^K)_{K \in \bbN^*} $ converges in law (and hence in probability) in $\bbD([0,T], \Rp{2})$ to the continuous function $y=(y_1,y_2) $ such that for all $t\geq0$,
	\begin{equation} \label{y}
	\left\lbrace \begin{array}{lll}
	y_1(t)&= x_1 \\
	y_2(t)&= \frac{\tau_1 x_1}{\tau_2} + \big( x_2 - \frac{\tau_1 x_1}{\tau_2} \big) \, e^{-\tau_2\,t}.
	\end{array}\right.
	\end{equation}
\end{theorem}

\begin{proof}
	By standard localization argument, use of Gronwall's Lemma and Doob's inequality, we easily prove (successively for the first and then for the second component) that for any $T>0$, 
	\be 
	\label{momentY}\displaystyle \sup_{K} \E{\sup_{t\in[0,T]} (Y^K_1(t)^2 +  Y^K_2(t)^2)} < \infty.\ee
	From \eqref{semi} and \eqref{crochet}, we can write
	\be
	\label{YK}
	Y^K_{1}(t)&=& Y^K_{1}(0) + \widehat M^K_{1}(t)\nonumber\\
	Y^K_{2}(t)&=& Y^K_{2}(0) + \tau_{1} \int_{0}^t Y^K_{1}(u) du -\tau_{2}   \int_{0}^t Y^K_{2}(u) du + \widehat M^K_{2}(t),
	\ee
	where $\widehat  M^K_{1}$ and $\widehat  M^K_{2}$ are two square-integrable martingales satisfying
	\be
	\label{MYK}
	\langle \widehat M^K_{1} \rangle_{t} &=& {\tau_{1}\over K^{1-\gamma_{2}}} \int_{0}^t Y^K_{1}(u) du,\nonumber\\
	\langle \widehat M^K_{2} \rangle_{t} &=& {2\tau_{1}\over K^{1+\gamma_{2}}} \int_{0}^t Y^K_{1}(u) du +{\tau_{2}\over K} \int_{0}^t Y^K_{2}(u) du,\nonumber\\
	\langle \widehat M^K_{1}, \widehat M^K_{2}  \rangle_{t} &=& - {\tau_{1}\over K} \int_{0}^t Y^K_{1}(u) du.
	\ee
	It is very standard to prove that the sequence  of laws of $(Y^K)$ is tight (using the moment estimates \eqref{momentY})  and that the martingale parts go to $0$. The result follows  using the method summarized for example  in \cite{CoursMeleardVincent}. Each limiting value is  proved to only charge   the subset of continuous functions. Then introducing
	$$ \phi_t(y)=  \left(\begin{array}{c}
	y_1(t) - y_1(0) \\  y_2(t) - y_2(0) - \int_0^t \big(\tau_1 y_1(s)- \tau_2 y_2(s)\big) ds   \end{array} \right)$$
	and using the uniform integrability of the sequence $(\phi_t(Y^K))_K$, deduced from \eqref{momentY},  we identify the limit as the unique continuous solution $y$  of the  deterministic system defined by $y(0)=(x_1,x_2)$ and 
	\begin{align*}
	\forall t \geq 0, \quad \phi_t(y)=0.
	\end{align*}
	That concludes the proof.
\end{proof}

\bi 
\begin{rema} Since $\gga_2<\gga_3$, the time scale $K^{\gga_2}$ is not large enough to observe the dynamics of the third component. 
	The next proposition shows  that at such a time scale, the third component converges to a trivial value.
\end{rema}

\begin{prop}
	Under the same hypotheses as in Theorem \ref{thY}, we assume  furthermore  that there exists $x_3 \in \Rp{}$ such that the sequence $(\frac{N^K_3(0)}{K^{1+\gga_2+\gga_{3}}})_{K \in \bbN^*}$ converges in law to $x_3$ when $K$ tends to infinity and such  that $$ \quad \sup_{K} \E{(\frac{N^K_3(0)}{K^{1+\gga_2+\gga_{3}}})^2} < \infty.$$ Then for each $T>0$, the sequence $(\frac{N^K_3(.K^{\gamma_2})}{K^{1+\gga_2+\gga_{3}}})_{K \in \bbN^*} $ converges  in probability in $\bbD([0,T], \Rp{})$ to $x_3$.
\end{prop}

\begin{proof}
	 Following \eqref{semi} and \eqref{crochet}, let us write the semimartingale decomposition of the process 
	$\ Y^K_{3} = \frac{N^K_3(.K^{\gamma_2})}{K^{1+\gga_2+\gga_{3}}}$.
	We have for any $t\le T$,
	\be
	Y^K_{3}(t)  =  Y^K_{3}(0)+ 2 \tau_{2} p^D_{2} K^{\gamma_{2}-\gamma_{3}} \int_{0}^t Y^K_{2}(s) ds - \tau_{3}K^{\gamma_{2}-\gamma_{3}} \int_{0}^t Y^K_{3}(s) ds + \widehat M^K_{3}(t),
	\ee
	where $\widehat M^K_{3}$ is a square-integrable martingale such that 
	$$\langle \widehat M^K_{3}\rangle_{t}= 2 \tau_{2} p^D_{2} K^{\gamma_{2}-\gamma_{3}} \int_{0}^t Y^K_{2}(s) ds + \tau_{3}K^{\gamma_{2}-\gamma_{3}} \int_{0}^t Y^K_{3}(s) ds .$$
	Let us recall that $\gamma_{2}<\gamma_{3}$, which makes $K^{\gamma_{2}-\gamma_{3}}$ tends to $0$ when $K$ tends to infinity. 
	
	Using Theorem \ref{thY}, we know that $Y^K_{2}$ converges to the continuous function $y_{2}$. By standard tightness argument, one can easily deduce that the process  $Y^K_{3}$ converges in probability  to $x_3$, on any finite time interval.
\end{proof}

\subsection{Asymptotic behavior at a time-scale of order $K^{\gamma_{3}}$}

\bi In order to catch the long time dynamics of the third component we will study the process $N^K$ on the time scale $tK^{\gga_3}$. To this end, let us introduce the jump process $Z^K$ defined  for all $t \geq0$ by 

$$ Z^K(t)=(\frac{N^K_1(t \, K^{\gga_3})}{K}, \frac{N^K_2(t\, K^{\gga_3})}{K^{1+\gga_2}},\frac{N^K_3(t\, K^{\gga_3})}{K^{1+\gga_2+\gga_3}}). $$

\me Note that  we still  have  $$Z^K(0)=Y^K(0)=X^K(0).$$

\me At this time scale, the second component has time to reach the equilibrium of its deterministic approximation by an average procedure. By an adaptation of the proof in \cite{lea2014} to this specific framework, we are able to prove the following theorem.
\begin{theorem} \label{thZ}		
	
	Assume that there exists  $(x_1,x_2,x_3) \in \Rp{3}$ such that the sequence $(Z^K(0))_{K \in \bbN^*}$ converges in law to $(x_1,x_2,x_3)$ when $K$ tends to infinity and such  that \be
	\label{hypZ}\displaystyle \sup_{K} \E{ Z^K_1(0)} < +\infty, \quad \sup_{K} \E{ Z^K_2(0)} < +\infty \quad \text{ and } \quad \sup_{K} \E{ Z^K_3(0)} < +\infty.\ee
	Let $\Gamma^K_2$ be the $l_m(\Rp{})$-valued random variable given by \be
	\Gamma^K_2([0,t]\times B)= \int_0^t \1{B}(Z^K_2(s))ds.
	\ee 
	Then for all $T>0$, the sequence $(Z^K_1,\Gamma^K_2,Z^K_3)_{K \in \bbN^*}$ converges in law in $\DT{\Rp{}} \times l_m(\Rp{})\times \DT{\Rp{}} $ to $(z_1, \delta_{z_2^*}(dz_2)\,ds, z_3 )$. The functions $z_1$ and $z_3$ are defined  for all $t\leq T$ by 
	\begin{equation}
	\left\lbrace \begin{array}{lll}
	\label{z}
	z_1(t)&= x_1 \\ \\
	z_3(t)&=  \begin{large}\frac{ \tau_2 }{\tau_3} z^*_2 + \big( x_3 - \frac{ \tau_2 }{\tau_3}z^*_2 \big) \, e^{-\tau_3\,t}  \end{large}
	\end{array}\right.
	\end{equation}and $z_2^*$ is the value of $y_{2}$ at infinity: $$z_2^*=\frac{\tau_1 x_1}{\tau_2}.$$
	
\end{theorem}

\me Let us first state a lemma in which all moment estimates are gathered.
\begin{lemme} \label{momentZ}
	Under Assumption \eqref{hypZ}, we obtain
	\ben
	\displaystyle \sup_{K} \E{ \sup_{t \in [0,T]} Z^K_1(t)} < +\infty \ ;\ \quad   \forall \, t>0\quad  \displaystyle \sup_{K} \E{\int_{0}^{t} Z^K_2(s) \, ds} < +\infty \ ;\ 
	\een
	and $$\displaystyle \sup_{K} \E{\sup_{t \in [0,T]} Z^K_3(t)} < +\infty.$$
\end{lemme}

\begin{proof}[Proof of Lemma \ref{momentZ}]
	The first and third estimates are obtained by usual arguments (localization, Doob's inequality and Gronwall's Lemma). 
	Let us focus on the second one. 
	
	\me  By positivity and definition of the process $Z^K_2$, we have for any $t>0$
	\be
	\label{zk2}Z^K_2(t)= Z^K_2(0) + \tau_{1}K^{\gamma_{3}-\gamma_{2}} \int_{0}^t Z^K_1(s) ds -  \tau_{2}K^{\gamma_{3}- \gamma_{2}} \int_{0}^t Z^K_2(s) ds + \widetilde M^K_{2}(t),\ee
	the latter term being a square-integrable martingale and
	\be
	\label{tildeM2}
	\langle  \widetilde M^K_{2} \rangle_{t} = \frac{1}{K^{1+\gamma_{2}}} \Big( 2 \tau_{1} K^{\gamma_{3}-\gamma_{2}}\int_{0}^t Z^K_{1}(s) ds + \tau_{2} K^{\gamma_{3}} \int_{0}^t Z^K_{2}(s) ds\Big).\ee
	In particular,
	$$\E{\int_0^t Z^K_2(s) \, ds} = {1\over \tau_{2}} K^{\gamma_{2}-\gamma_{3}}  \Big(\E{ Z^K_2(0)}-\E{ Z^K_2(t)} \Big) +  {\tau_1\over \tau_{2}}\int_0^t \E{Z^K_1(s)} \,ds .$$ 
	Assumptions	\eqref{hypZ}	and  Lemma \ref{preliminary} ensure that the first term goes to $0$ as $K$ tends to infinity and the third  term is bounded uniformly in $K$. That allows to conclude.				
\end{proof}

\begin{proof}[Proof of Theorem \ref{thZ}]	 
	Let $\Gamma^K$ be the occupation measure of $Z^K$, a random measure belonging to  the space $l_m(\Rp{})$ of positive measures on $[0,\infty) \times \Rp{} $ with mass $t$ on$ [0,t]\times \Rp{}$  and defined  for all Borelian set  $B$  and for $t>0$ by
	$$\Gamma^K([0,t] \times B)=\int_0^t \1{B}(Z^K(s))\, ds.$$
	
	%Let $\phi : \Rp{3} \to [1,+\infty)$ be the locally bounded function $$\forall x \in \Rp{3}, \quad \phi(x)=1+x_1+x_2+x_3.$$   Lemma \ref{momentZ} yields  \be \label{control}\forall t, \quad \sup_{K} \E{\int_0^t \phi (Z^K(s)) ds} < +\infty. \ee
	
	\me	Using  Lemma 2.9 of \cite{lea2014} (cf. Appendix), we obtain that $(\Gamma^K)_K$ is relatively compact in $l_m(\Rp{})$ endowed with a weak topology generated by the  class of test functions defined in \eqref{testfunction} .
	
	\me 	 Let us denote by $\Gamma\in l_m(\Rp{3})$ a limiting value. Using \cite{kurtz92} Lemma 1.4, one can show that there exists a $\cP(\Rp{3})$-valued process $\gamma_s$ such that $$\Gamma(dz \times ds)= \gamma_s(dz) \, ds.$$ 
	
	\me Let us now introduce  the function $F^K$  $$\forall z \in \Rp{3}, \quad F^K(z) = (1+{1\over K^{\gamma_{2}}})\,\tau_2 z_2 - \tau_3 \, z_3.$$
	Then for all $t \geq 0$,
	\begin{align}\label{zk3} \quad Z^K_3(t) &= Z^K_3(0) + (1+{1\over K^{\gamma_{2}}})\tau_2 \int_0^t Z^K_2(s)ds - \tau_3 \, \int_0^t Z^K_3(s) ds+ \widetilde M^K_3(t)  \\
	&=Z^K_3(0) + \int_0^t F^K(Z^K(s)) \, ds + \widetilde M^K_3(t)\nonumber \\
	\label{zk1}Z^K_1(t) &= Z^K_1(0) + \widetilde M^K_1(t)
	\end{align}
	with $(\widetilde M^K_1,\widetilde M^K_3)$ independent martingales and such that for all $t \geq 0$, \begin{align} \label{crochet1}<\widetilde M^K_1>_t =& K^{\gga_3-1}\, \int_{0}^{t} 2\tau_1\, Z^K_1(s) \,ds\\
	\label{crochet3} <\widetilde M^K_3>_t =& K^{-(1+\gga_2+\gga_3)}\,  \int_{0}^{t} \big( 2(1+{1\over K^{\gamma_{2}}})\tau_2 \, Z^K_2(s) + \tau_3 Z^K_3(s) \big)\,ds. \end{align}
	
	\me  By usual arguments involving  Lemma \ref{momentZ} one can prove that the sequences of  processes $(Z^K_3)_{K}$ and   $(Z^K_1)_{K}$ are uniformly tight  in $\DT{\Rp{}}$. Let us also note that the distributions of any limiting value only charge processes with a.s. continuous trajectories. 
	
	\me 	Furthermore by Doob's inequality \begin{align*}\E{\sup_{t \leq T} \va{\widetilde M^K_3(t)}^2}\leq 4 \,\E{<\widetilde M^K_3>_T} .
	\end{align*}
	Using (\ref{crochet3}) and Lemma \ref{momentZ}, we obtain that 
	$\ \lim_{K \to \infty} \E{ \sup_{t \in [0,T]} \va{ \widetilde M^K_3(t)}^2} =0$ and  a similar  property for  $\widetilde M^K_1$ since $\gamma_{3}<1$. 
	Then, we deduce from Markov's inequality, that the processes $(\widetilde M^K_3)_{K}$ and $(\widetilde M^K_1)_{K}$ converge in probability for the uniform norm to 0. Hence they converge in law in $\DT{\Rp{}}$ to $0$. 
	
	\bi Adding all these asymptotic behaviors, we deduce that there exists a subsequence  of $(Z^K_1,Z^K_3)$  converging in law in $\bbD([0,T],\Rp{2})$ to the deterministic limit $(z_1, Z_3^{\infty})$ defined for all $t \geq 0$ by \begin{align*}\, \quad z_1(t) &=  x_1 \\
	Z_3^{\infty}(t) &=  x_3 + \int_{\Rp{3}\times [0,t]} (\,\tau_2 z_2 - \tau_3 \, z_3) \gamma_{s}(dz)  ds.
	\end{align*}

	\bi	Then by convergence of $(Z^K_1)_{K}$, $(Z^K_3)_{K}$ and $\Gamma^K$, one can easily deduce that  $$\gamma_s(dz) = \delta_{\bar z_1} (dz_1) \delta_{Z^{\infty}_3(s)}(dz_3) \, \tilde \gamma_s(z_{1},z_{3}, dz_2). $$
	
	\bi We have now to identify these measures  $\tilde \gamma_{s}, \ s\in [0,T]$. 	
	
	\me Let us write the generator of the process $Z^K_2$. For  $h \in C^\infty_{c}$	, it is given for $z \in \Rp{3}$ by
	\begin{align*}  \cL^K_2(h)(z) = &\big( h(z_2+2\, K^{-(1+\gga_2)}) - h(z_2) \big) \, \frac{\tau_1}{2} \, z_1 \, K^{1+\gga_3}\\ & + \big( h(z_2+\, K^{-(1+\gga_2)}) - h(z_2) \big) \,p^R_2 \tau_2 \, z_2 \, K^{1+\gga_2+\gga_3}\\ & + \big( h(z_2-\, K^{-(1+\gga_2)}) - h(z_2) \big) \, p^D_2 \tau_2 \, z_2 \, K^{1+\gga_2+\gga_3} . \end{align*}
	
	\me Let us introduce the function $g$ defined for $z \in \Rp{3}$ by $$ \quad g(z) = \,z_1 \tau_1 - \tau_2 z_2. $$
	Then,  by a Taylor expansion, we obtain that
	\be\label{CVgen}\forall h \in C^\infty_{c},\quad \lim_{K \to \infty} \sup_{z \in \Rp{}} \va{K^{\gga_2-\gga_3} \cL^K_2(h)(z) - g(z)h'(z_2)}=0.\ee
	
	\bi Using \eqref{tildeM2},  Lemma \ref{lea} and the same arguments as above, we obtain that the  sequence of processes $  \Big(K^{\gga_2 -\gga_3}\,\Big(h(Z^K_2(t)) - h(Z^K_2(0)) - \int_0^t \cL_2^K(h)(Z^K(s)) \, ds\Big), \quad t\in [0,T] \Big)$ converges in law in $\bbD([0,T],\bbR)$ to $0$.
	In the other hand, using that $h$ is bounded and  \eqref{CVgen} and $g\, h'\in C_{b}$ (since $h$ has compact support), we easily see that  this sequence also converges to 
	$$ - \int_0^t ( \tau_1 x_1 - \tau_2 z_2) h'(z_2) \,  \gamma_s(z_{1}, z_{3}, dz_2) \, ds.$$ 	\me We deduce that for any $t\in[0,T]$, for any  $h \in C^\infty_{c}$
	$$ \int_0^t ( \tau_1 x_1 - \tau_2 z_2) h'(z_2) \,\tilde \gamma_s(z_{1}, z_{3}, dz_2) \, ds = 0.$$
	It implies that 	  $$\tilde \gamma_s(z_{1}, z_{3}, dz_2) = \delta_{ \frac{ \tau_1  x_1}{ \tau_2}}(dz_2).$$
	
	\me To end the proof, we  solve the equation satisfied by  $Z^{\infty}_3(t)$  and obtain that $$ \forall t \geq 0, \quad Z^{\infty}_3(t)= \frac{ \tau_1 x_1}{\tau_3} + \big( x_3 - \frac{\tau_1 x_1}{\tau_3} \big) \, e^{-\tau_3\,t}.$$
	Hence we have uniquely identified the limit of any converging subsequence.  That ends the proof.
	
\end{proof}

\section{Amplified fluctuations}
In this section, we will quantify the large fluctuations observed on the simulations. As pointed out above, each component has its own typical size and time scale. Hence,  we will study separately the fluctuations of the second and third types.  Classical results easily imply  that	 the first component behaves as a Brownian motion: for large $K$, for all $t\in[0,T]$,
$$N^K_1(t ) \sim x_1\,K + \sqrt{K}\,\sqrt{\tau_{1} x_1 } \, B_{t}.$$
%which explains the simulations of Figure 1. 
The originality of our results concerns the large fluctuations of the two last types due to the amplification of these first type fluctuations. 

\subsection{The large fluctuations of the second type}
\bi  As seen in Subsection 3.1,  the typical size scale (respectively time scale) of the second type is $K^{1+\gga_2}$ (respectively $K^{\gga_2}$) and  the first order asymptotics relative to this time scale is given by the function $y$ defined by \eqref{y}.
We are also able to give an expansion of the process  at  the second and  third orders  on such time scale.

\begin{theorem} \label{thU}
	Let define the sequence $(U^K)_K$ by $$ \forall t \geq 0, \quad U^K(t)=K^{(1-\gga_2)/2} \big( Y^K(t)-y(t) \big).$$
	(i)	Assume that there exists $U_0=(U_0^{(1)},U_0^{(2)})\in \bbR^2$ such that  $(U^K(0))_{K \in \bbN^*}$ converges in law to $U_0$ and that \be
	\label{CI}\displaystyle \sup_{K} \E{ U^K_1(0)^2+ U^K_2(0)^2} < +\infty.\ee
	Then for each $T>0$, the sequence $(U^K)_{K \in \bbN^*} $ converges in law in $\bbD([0,T], \bbR^2)$ to the process $U=(U_1,U_2)$ defined  for all $t\geq0$ by
	\be U_1(t)&=&U_0^{(1)} + \sqrt{\tau_1 \, x_1 } \,B_1(t), \\
	U_2(t)&=& U_0^{(2)} + \tau_1 \int_0^t U_1(s) ds - \tau_2 \int_0^t U_2(s) \, ds,\ee where $B_1$ is a standard Brownian motion.
	
	\bi (ii)	Furthermore, the sequence $(W^K_2)_{K \in \bbN^*}$ defined by $$\forall t \in [0,T], \quad W^K_2(t) = K^{\gga_2/2} \; \big[ U^K_2(t) - U_0^{(2)} - \tau_1 \int_0^t U_1^K(s) ds + \tau_2 \int_0^t U^K_2(s)ds \big],$$ converges in law in $\bbD([0,T], \bbR)$, for each $T>0$, to the process $(\sqrt{\tau_2 y_2(t)} \, B_2(t), \quad t \in [0,T])$ where $B_2$ is a standard Brownian motion independent of the process $B_1$.
\end{theorem}

From this theorem, we can deduce the following expansion, which quantifies the large waves of  fluctuations. Assuming that $U_0$ is equal to zero,  we obtain for all $t$ and large $K$, 

$$N_2^K(t) \sim K^{1+\gga_2} y_2(t \, K^{-\gga_2}) + K^{(1+ 3\gga_2)/2} U_2(t \, K^{-\gga_2}) +  K^{(1+ 2\gga_2)/2} \sqrt{\tau_2 \,y_2(t\,K^{-\gga_2}) } \,B_2(t \, K^{-\gga_2}) $$
where $$\forall t, \quad U_2(t) = \tau_1 \sqrt{\tau_1 x_1} \int_0^t B_1(s)\, ds - \tau_2 \int_0^t U_2(s)\, ds$$ and $B_1$, $B_2$ are independent Brownian motions.

\me
\begin{proof}[Proof of Theorem \ref{thU}.]
	(i) First we deduce from \eqref{CI} with similar arguments as above that
	\be
	\label{momentU}\displaystyle \sup_{K} \E{\sup_{t\in[0,T]} (U^K_1(t)^2+ U^K_2(t)^2)} < +\infty.\ee
	The tightness of the families $(\displaystyle \sup_{t \leq T } \va{ U^K_1(t) })_{K}$ and $(\displaystyle \sup_{t \leq T } \va{ U^K_2(t) })_{K}$  immediately follows.\\
	
	\bi We consider   the semi-martingale decomposition of $(U^K)$ and  write 
	$$U^K_i(t) = U^K_i(0) + A^K_i(t) + K^{\frac{1-\gga_2}{2}}\widehat M^K_i(t),$$ where $\widehat M^K$ has been defined in \eqref{YK}, $A^K_{1}
	= 0$ and $A^K_2(t)= \tau_1 \int_0^t  U^K_1(s) \,ds - \tau_2 \int_0^t U_2^K(s) \, ds$.
	
	\bi	Thanks to the above moment estimates, it is almost immediate to prove that the finite variation processes $<K^{\frac{1-\gga_2}{2}}\widehat M^K_i>$ and $A^K_2$ satisfy the Aldous condition. Thanks to  Aldous and Rebolledo criteria (see \cite{JoffeMetivier} and \cite{CoursMeleardVincent}) , the uniform tightness of $\cL(U^K)$ in $\cP(\bbD([0,T],\bbR^2))$ follows.
	
	\bi We denote by simplicity by the same notation $(U^K)_K$ a subsequence  converging in law in $\bbD([0,T],\bbR^2)$. Let  $Q$ be the limiting value of  $(\cL(U^K))_K$. It is easy to observe that 
	$$\displaystyle \sup_{t\in [0,T]} \p{\Delta U^K(t)} \leq 2K^{-(1+\gga_2)/2}.$$
	Therefore, by continuity of  the mapping $x \to \displaystyle \sup_{t\in [0,T]} \p{\Delta x(t)}$  from $\bbD([0,T],\bbR^2)$ into $\Rp{}$,   the probability measure $Q$ only charges the processes with continuous paths.
	
	\bi The extended generator of $U^K$ is defined for functions $f \in C^2_b(\bbR^2,\bbR)$ as: $ \forall u \in \bbR^2$,
	\be\begin{split} \label{gu} \cL^K(f,t)(u)& = \big( f(u_1+K^{-(1+\gga_2)/2},u_2) - f(u) \big) \, \frac{\tau_1}{2} \, K^{1+\gga_2} (K^{-(1-\gga_2)/2}u_1 + x_1) \\ +& \big( f(u_1-\, K^{-(1+\gga_2)/2},u_2+2 K^{-(1+3\gga_2)/2}) - f(u) \big) \,  \frac{\tau_1}{2} \, K^{1+\gga_2} (K^{-(1-\gga_2)/2}u_1 + x_1)\\ +& \big( f(u_1,u_2+ K^{-(1+3\gga_2)/2}) - f(u) \big) \, p^R_2 \,\tau_2 \, K^{1+2\gga_2} (K^{-(1-\gga_2)/2}u_2 + y_2(t))\\ +& \big( f(u_1,u_2- K^{-(1+3\gga_2)/2}) - f(u) \big) \, p^D_2 \,\tau_2 \, K^{1+2\gga_2} (K^{-(1-\gga_2)/2}u_2 + y_2(t))\\-& K^{(1-\gga_2)/2} \partial_2f(u) \big(\tau_1 x_1 - \tau_2 y_2(t) \big).  \end{split}\ee
	
	\me By a Taylor's expansion, we easily obtain that $\forall f \in C^{2}_b(\bbR^2,\bbR)$,
	\begin{equation}
	\label{lk1}
	\lim_{K \to \infty} \sup_{(u,t) \in \bbR\times \Rp{}} \big| \cL^K(f,t)(u) - (\frac{\tau_1}{2} x_1 \partial_1^2f(u) + (\tau_1 u_1 -\tau_2 u_2)\partial_2f(u)) \big|=0.\end{equation}

	\me In the other hand, let us define, for $f \in C^{2}_b(\bbR^2,\bbR)$, $u \in \bbD([0,T],\bbR^2)$ and $t \in [0,T]$, the function $\xi^{K,f}_t$ by $$\xi^{K,f}_t(u) = f(u_t) - f(u_0)- \int_0^t \cL^K(f,s)(u_s)\,ds.$$
	Then, by (\ref{gu}), Dynkin's formula and \eqref{momentU}, we can easily prove that the processes  $(\xi_t^{K,f}(U^K))_K$ are  uniformly integrable martingales.
	
	Therefore by standard arguments (cf. \cite{EthierKurtz}, \cite{CoursMeleardVincent}),  the limiting process under $Q$  is  continuous and satisfies the following martingale problem: $$ \forall f \in C^2_b(\bbR^2,\bbR), \quad f(U(t)) - f(U_0) - \int_0^t  (\frac{\tau_1}{2} x_1\partial_1^2f(U(s)) + (\tau_1 U_1(s) -\tau_2 U_2(s))\partial_2f(U(s)))\, ds$$ is a martingale.
	
	\me 	We conclude using a representation theorem (cf. \cite{Ikeda-Watanabe} p.84)  that for each $T>0$, the sequence $(U^K)_{K \in \bbN^*} $ converges in law in $\bbD([0,T], \bbR^2)$ to the  process $U=(U_1,U_2)$, unique solution of the following stochastic differential system:  for all $t\in [0;T]$,
	\ben U_1(t)&=&U_0^{(1)} + \sqrt{\tau_1 \, x_1} \,B_1(t),\\
	U_2(t)&=& U_0^{(2)} + \tau_1 \int_0^t U_1(s) ds - \tau_2 \int_0^t U_2(s) \, ds,\een with $B_1$ a Brownian motion.
	
	\bi (ii) Let us now expand the second component to the next order. We deduce from \eqref{YK} that 
	\ben U^K_1(.)- U^K(0)&=& K^{(1-\gga_2)/2} \,\widehat M^K_1(.),\\
	W^K_2(.)&= &\sqrt{K} \, \widehat M^K_2(.).\een
	
	\me Using \eqref{MYK}, \eqref{momentY} and applying Theorem 7.1.4 of \cite{EthierKurtz}, we conclude the proof.
\end{proof}

\subsection{The large fluctuations of the third type}

Let us now study the fluctuation process associated with the  largest fluctuation scale of the third component. We have seen in Theorem \ref{thZ} that at the time scale $K^{\gamma_{3}}$, the size of the population process in the third compartment is of order of magnitude $K^{1+\gamma_{2} +\gamma_{3}}$. In an usual setting, the Central Limit Theorem would lead to fluctuations of order $K^{(1+\gamma_{2} +\gamma_{3})/2}$. We will see in the next theorem that they are of order $K^{(1+2\gamma_{2} +3\gamma_{3})/ 2}\gg K^{(1+\gamma_{2} +\gamma_{3})/2}$, since amplified by the  fluctuations of the first compartment. 

\bi Using  \eqref{z}, \eqref{zk2} and \eqref{zk3}, we know that  for all $t\ge 0$,
\begin{align}\label{flucz3}
\big( Z_3^K(t)-z_3(t) \big) &= \big( Z_3^K(0)-z_3(0) \big) + \tau_2 \int_0^t \big( Z_2^K(s)-z_2^* \big) ds - \tau_3 \, \int_0^t \big( Z_3^K(s)-z_3(s) \big) ds \nonumber \\ & \quad  + \tau_2 \frac{1}{ K^{\gga_2}} \int_0^t \big( Z_2^K(s)-z_2^* \big) ds +\widetilde M^K_3(t)
\end{align}
where 
\begin{align}\label{flucz2}
\big( Z_2^K(t)-z_2^* \big) &= \big( Z_2^K(0)-z_2^* \big) + K^{\gga_3-\gga_2}\tau_1 \int_0^t \big( Z_1^K(s)-x_1 \big) ds \nonumber\\ & \quad - K^{\gga_3-\gga_2}\tau_2 \, \int_0^t \big( Z_2^K(s)-z_2^* \big) ds   +\widetilde M^K_{2}(t).
\end{align}

\me  Our goal is to quantify the effect of the first component fluctuations on the dynamics of the third component. Considering the  expressions of the  martingale quadratic variation  \eqref{crochet1}  imposes the choice of the  rescaling parameter   $K^{(1-\gga_3)/2}$ in front of $\big( Z_1^K(t)-x_1 \big) $.  We will see that to keep the effect of the first component on the third component, we need to rescale $\big( Z_2^K(t)- Z_2^K(0)\big) $ by $\frac{K^{(1-\gga_3)/2}}{K^{\gga_3-\gga_2}}$ and $\big( Z_3^K(t)-z_3(t) \big) $ by $K^{(1-\gga_3)/2}$.

\me Let us now state the main theorem of this section.

\begin{theorem} \label{thV}
	Let us define  the three processes $$ \forall t \geq 0, \quad \left\lbrace \begin{array}{lll}
	V_1^K(t)&=K^{(1-\gga_3)/2} \big( Z_1^K(t)-x_1 \big) \\ \\
	V_2^K(t)&= \frac{K^{(1-\gga_3)/2}}{K^{\gga_3-\gga_2}} \big( Z_2^K(t)- Z_2^K(0) \big)  \\ \\
	V_3^K(t)&=K^{(1-\gga_3)/2} \big( Z_3^K(t)-z_3(t) \big)  
	\end{array}\right..$$
	Let us assume that there exists $V_0=(V_0^{(1)},V_0^{(3)})$ a $\bbR^2$-valued random vector such that the sequence $(V_1^K(0),V_3^K(0))_{K \in \bbN^*}$ converges in law to $V_0$ and such that \be
	\label{civz} \displaystyle \sup_{K} \E{ V^K_1(0)^4 }< +\infty \quad ; \quad  \sup_{K} \E{Z_2^K(0)^2} <+\infty.\ee
	\be
	\label{civ3}
	\sup_{K} \E{ \va{V^K_3(0)}} <+ \infty. \ee
	
	Then for all $T>0$, the sequence $(V^K_1,V^K_3)_{K \in \bbN^*}$ converges in law in $\DT{\bbR^2}$ to $(V_1, V_3 )$ such that for all $t$,
	$$V_1(t)=V_0^{(1)} + \sqrt{\tau_1 \, x_1}W_1(t) $$
	$$V_3(t)= V_0^{(3)} + \tau_1 \int_0^t V_1(s) ds - \tau_3 \int_0^t V_3(s) \, ds,$$ where $W_1$ is a standard Brownian motion.
	
\end{theorem}
\me
Let us interpret this result in terms of  fluctuations. Assuming that the initial vector $V_0$ is equal to zero, we obtain that for any $t$ and large $K$, 

\be
\label{magnitude}N_3^K(t) \sim K^{1+\gga_2+\gga_3} \,z_3(t \, K^{-\gga_3}) + K^{(1+2\gga_2+3\gga_3)/2} \, V_3(t \, K^{-\gga_3}) \ee
where $$\forall t, \quad V_3(t) = \tau_1 \sqrt{\tau_1 x_1} \int_0^t W_1(s)\, ds - \tau_3 \int_0^t V_3(s)\, ds $$ and $W_1$ is a standard Brownian motion.

\me The order of magnitude   appearing in the  fluctuation second order term  \eqref{magnitude} summarizes the cumulative effects of the third dynamics driven by the fluctuations of the first level. That can explain
the exceptionally large fluctuations observed for the size of terminal cells populations, in hematopoietic systems.

\me As a first step in the proof of Theorem \ref{thV}, we will prove that  the sequence of processes  $(V^K_2)_{K}$ converges to $0$  uniformly in $\mathbb{L}^2$ on any finite time interval.
\begin{prop} \label{propV2}
	Under the assumption \eqref{civz}, we obtain \be
	\label{V2}\forall T \geq 0, \quad  \lim_{K \to \infty} \E{\sup_{t \in [0,T]} V^K_2(t)^2} =0.
	\ee
\end{prop}

\begin{proof}
	Using \eqref{zk2}, we obtain that
	\begin{align*}
	\forall t \geq 0, \quad V^K_2(t) &= K^{\frac{1-3\gga_3}{2}+\gga_2} \big( Z^K_2(t) - Z^K_2(0) \big) \\
	& = \tau_1 \int_0^t V_1^K(s) \,ds - \tau_2\,K^{\gga_3-\gga_2}  \int_0^t V^K_2(s) \, ds + \cR^K(t),
	\end{align*}
	where $\cR^K$ is the square-integrable martingale  defined by 
	\begin{equation}
	\label{cr}
	\forall t \geq 0, \quad \cR^K_t = K^{\frac{1-3\gga_3}{2}+\gga_2} \widetilde M^K_2(t)
	\end{equation}
	and satisfying  \begin{equation}
	\label{rk}
	\forall t, \quad \langle \cR^K \rangle_t =\,  K^{-2\gga_3} \Big(\tau_1\,\int_{0}^{t}  Z^K_1(s) \, ds +  \tau_2 \,K^{\gga_2} \int_{0}^{t}  Z^K_{2}(s) \, ds \Big). \end{equation}
	Let us first show that  $$\forall t>0, \quad \sup_{K} \E{ \int_0^t V^K_2(s)^4 \,ds } < \infty.$$
	It\^o's formula immediately implies that $\forall t \geq 0$, 
	\begin{align*}
	V^K_2(t)^4 &=  V^K_2(0)^4 + 4 \, \int_0^t V_2^K(s)^3 \,d\cR^K_s +  4 \, \int_0^t V_2^K(s)^3 \,\big( \tau_1 V_1^K(s) -K^{\gga_3-\gga_2} \tau_2 V^K_2(s)\big) \, ds \\ & \quad \quad + 6 \, \int_0^t V_2^K(s)^2 \,d \langle \cR^K \rangle_s.
	\end{align*}
	
	By standard localization arguments,  we prove using \eqref{civz}  that for any $t\ge 0$,
	\begin{equation} \label{controlv1}
	\forall t \geq 0, \quad \sup_{K} \E{ \int_0^t V^K_1(s)^4 \,ds } < \infty.\end{equation}
	
	Let us now  introduce the stopping time
	$$T_{n}= \inf\{t \ge 0, | V^K_2(t)| \ge n\}.$$
	
	Then, applying the following inequality 	\begin{align} 
	\label{astuceS}
	4 \big( \tau_1 v_1 v_2 - K^{\gga_3-\gga_2} \tau_2 v_2^2 \big) &= 4 \, \tau_2 K^{\gga_3-\gga_2} \big( [\frac{ v_1 \tau_1}{2 \tau_2 \,K^{\gga_3-\gga_2} }]^2 -  [v_2 - \frac{ v_1 \tau_1}{2 \tau_2 \,K^{\gga_3-\gga_2} } ]^2 \big)  \nonumber \\ &\leq  \,  \frac{ v_1^2 \tau_1^2}{ \tau_2 \,K^{\gga_3-\gga_2} }
	\end{align} 
	to $ v_{1} = V_1^K(s)$ and $v_{2}=  V_2^K(s)$,
	we obtain the following upper-bound.
	
	\begin{align}
	\label{majv2}
	\E{V^K_2({t\wedge T_{n}})^4} 
	\leq  \E{V^K_2(0)^4} &+  \, \frac{\tau_1^2}{ \tau_2 \,K^{\gga_3-\gga_2} } \int_0^{t\wedge T_{n}} \E{V_2^K(s)^2 \, V_1^K(s)^2} ds \nonumber \\ &+ 6 \, \E{\int_0^{t\wedge T_{n}} V_2^K(s)^2 \,d\langle \cR^K \rangle_s }.
	%\\	&\leq  \E{V^K_2(0)^4} +  \frac{ \tau_1^2}{  2 \, \tau_2 \, K^{\gga_3-\gga_2} } ( \int_0^t \E{V_2^K(s)^4} \,ds +\int_0^t \E{V_1^K(s)^4} ds)  \\ & \quad \quad + 6 \, \E{\int_0^t V_2^K(s)^2 \,d\langle \cR^K \rangle_s }.
	\end{align}
	Using (\ref{rk}), we obtain for all $t \in [0,T]$,
	$$
	\E{\int_0^{t\wedge T_{n}} V_2^K(s)^2 \,d\langle \cR^K \rangle_s } = K^{\gga_2-2\gga_3} \, \E{\int_0^{t\wedge T_{n}} V_2^K(s)^2 \,  \big( K^{-\gga_2}\tau_1 Z^K_1(s) + \tau_2 Z^K_2(s) \big) \,ds }.
	$$
	Writing  $Z_2^K$ in function of $V_2^K$, we find the following upper bound,
	\begin{align}
	\label{majrk}
	\E{\int_0^{t\wedge T_{n}} V_2^K(s)^2 \,d\langle \cR^K \rangle_s } & \leq  \big(  \tau_1 K^{-2\gga_3} + \tau_2 K^{-(1+\gga_3)/2} +\tau_2 K^{\gga_2-2\gga_3} \big) \int_0^{t\wedge T_{n}} \E{ V_2^K(s)^4 } \, ds \nonumber \\&+ \tau_1 K^{-2\gga_3}  \int_0^{t\wedge T_{n}}\E{Z^K_1(s)^2} \,ds  \\ &+ \tau_2 \big( K^{-(1+\gga_3)/2} + K^{\gga_2-2\gga_3} \big)(\E{Z_2^K(0)^2}+1)\,T. \nonumber
	\end{align}
	We deduce from  \eqref{majv2} using Lemma \ref{momentZ},  \eqref{civz} and  Gronwall's Lemma that $$\forall t \in [0,T], \quad \sup_{K} \E{ \int_0^{t\wedge T_{n}} V^K_2(s)^4 \,ds } < \infty.$$
	
	\bi Let us now come back to the proof of \eqref{V2}.  It\^o's formula yields
	\begin{align*}
	V^K_2(t)^2 =  2 \, \int_0^t V_2^K(s) \,d\cR^K_s +  2 \, \int_0^t V_2^K(s) \,\big( \tau_1 V_1^K(s) -K^{\gga_3-\gga_2} \tau_2 V^K_2(s)\big) \, ds +  \langle \cR^K \rangle_t.
	\end{align*}
	Therefore, using again  (\ref{astuceS}) and Doob's inequality, we obtain
	\begin{align*}
	\E{\sup_{t \in [0,T\wedge T_{n}]}V^K_2(t)^2} &\leq   8 \, \E{\int_0^{T\wedge T_{n}} V_2^K(s)^2 \,d\langle \cR^K \rangle_s}  +   \, \frac{ \tau_1^2}{ \tau_2 K^{\gga_3-\gga_2} } \E{\int_0^{T\wedge T_{n}} V_1^K(s)^2 \, ds} \\ & \quad \quad +  \E{\langle \cR^K \rangle_{T\wedge T_{n}}}.
	\end{align*}
	
	Finally, we deduce from Lemma \ref{momentZ}, \eqref{majrk}, \eqref{civz} 
	and  \eqref{controlv1}   that  for any $K$, $T_{n}$ tends almost surely to $+\infty$  and that 
	$$\forall T \geq 0, \quad  \lim_{K \to \infty} \E{\sup_{t \in [0,T]} V^K_2(t)^2} =0.$$
\end{proof}

\bi Let us now come back to the proof of Theorem \ref{thV}. It has been  inspired by the proof of the main result in \cite{lea2014}.
\begin{proof}[Proof of Theorem \ref{thV}.]
	Using similar convergence arguments as in Theorem \ref{thU} and  \eqref{civz}, we firstly observe that the sequence $(V^K_1)_K$ converges in law in $\DT{\bbR}$ to a continuous process $V_1$ defined by 
	\be
	\label{V1}\forall t, \quad V_1(t)=V_0^{(1)} + \sqrt{\tau_1 x_1} W_1(t)\ee with $W_1$ a standard Brownian motion.
	
	\me Let us recall that from  \eqref{flucz3}, \eqref{flucz2} and \eqref{cr},  that for all $t$,
	$$V^K_3(t) =\, V^K_3(0) +  \tau_2\,  \big(1+\frac{1}{K^{\gga_2}}\big) \, \int_{0}^{t} K^{(1-\gga_3)/2} (Z^K_2(s) -z_2^*)\, ds - \tau_3 \int_{0}^{t}  V^K_3(s) \,ds + K^{(1-\gga_3)/2} \widetilde M^K_3(t)$$
	with $$\tau_2 \int_{0}^{t}  K^{(1-\gga_3)/2} \big(Z^K_2(s)-z^*_2) \, ds = K^{(1-\gga_3)/2} \, \frac{Z^K_2(0) - Z^K_2(t)}{K^{\gga_3-\gga_2}} + \int_{0}^{t} \tau_1\, V^K_1(s) \, ds  + \cR^K_t.$$
	Hence, for all $t$,
	\begin{equation} 
	\label{v3simple}
	V^K_3(t) =\, V^K_3(0) + (1+\frac{1}{K^{\gga_2}}) \, \int_{0}^{t} \tau_1\,V^K_1(s) \, ds - \tau_3 \int_{0}^{t}  V^K_3(s) \,ds - (1+\frac{1}{K^{\gga_2}}) \, V^K_2(t) + \cM^K_t,
	\end{equation}
	where $\cM^K$ is the square-integrable martingale  $$\cM^K(t)= (1+\frac{1}{K^{\gga_2}}) \cR^K_t + K^{(1-\gga_3)/2} \widetilde M^K_3(t).$$
	We deduce from \eqref{crochet3}, \eqref{rk}, Lemma \ref{momentZ} and Doob's inequality, that
	\begin{equation}
	\lim_{K \to \infty} \E{ \sup_{t \in [0,T]} \va{\cM^K(t)}^2}=0.
	\end{equation}
	Then it turns out from Markov's inequality that the sequence $(\cM^K)_K$ converges in probability to $0$ for the uniform norm and hence $(\cM^K)_K$ converges in law in $\DT{\bbR}$ to $0$.\\
	
	Furthermore using  \eqref{v3simple}, \eqref{controlv1}, \eqref{civ3} and Proposition \ref{propV2} we obtain \begin{align*} 
	\sup_{K}\E{\sup_{t \leq T } \va{V^K_3(t)}} < \infty.
	\end{align*}
	We  are now able to prove the tightness of the family $(\displaystyle \sup_{t \leq T } \va{ V^K_3(t) })_{K}$. Indeed, les us introduce 
	stopping times $S$,$S'$ satisfying $S \leq S'\leq (S+\delta)\w T$, with $\delta >0$. Using  (\ref{v3simple}), we have  \begin{align*} 
	\bbP(\va{V^K_3(S')-V^K_3(S)}>\epsilon) &\leq \, \frac{1}{\epsilon} \, \E{ \va{V^K_3(S')-V^K_3(S)}}  \\&\leq \, \frac{1}{\epsilon} \, \E{ \va{V^K_3(S')-V^K_3(S)}^2}^{1/2}  \\ &\leq \,\frac{1}{\epsilon} \, \Big[ \, \delta \, T \, \Big( \tau_1(1+\frac{1}{K^{\gga_2}})  \, (\E{\sup_{t \leq T }V^K_1(t)^2} +1) + \tau_3  \E{\sup_{t \leq T }\va{V^K_3(t)}} \Big) \\ & \quad \quad + 4 \,\E{\sup_{t \leq T } \cM^K(t)^2} + 4\,p^D_2 \,\E{\sup_{t \leq T } V^K_2(t)^2} \, \Big]^{1/2}.
	\end{align*}
	Then from  \eqref{v3simple}, \eqref{controlv1}, \eqref{civ3} and Proposition \ref{propV2}, we deduce that Aldous conditions (see \cite{JoffeMetivier} and \cite{CoursMeleardVincent}) are satisfied and obtain the tightness of $(V^K_3)_K$.
	
	\me Finally, using Proposition \ref{propV2}, the convergence in law in $\DT{\bbR}$ of the processes $\cM^K$ and  $V^K_1$ respectively to  zero and $V_1$ and the convergence in law of $V^K_3(0)$ to $V_0^{(3)}$, we obtain  that  the sequence $(V^K_3)_K$  converges in law in $\DT{\bbR}$  to the process $V_3$, unique solution of the following SDE
	
	$$\forall t \in [0,T], \quad V_3(t)= V_0^{(3)} + \tau_1 \,(1+\frac{1}{K^{\gga_2}}) \,\int_0^t V_1(s) ds - \tau_3 \int_0^t V_3(s) \, ds,$$ where  $V_{1}$ has been defined in \eqref{V1}. That ends the proof.
\end{proof}

\section{Appendix}

\begin{lemme}[Lemma 2.9 of \cite{lea2014}]
	\label{lea}
	Let $V^N$ be a sequence of $\Rp{3}$-valued processes. We consider its   occupation measure defined for $D$ a Borelian set by $$\Gamma_N(D \times [0,t])= \int_0^t \1{D}(V^N(s))\, ds.$$ Let us assume  that there exists a function $\psi : \Rp{3} \to [1,\infty)$ locally bounded such that $\lim_{v \to +\infty} \psi(v)= +\infty$ and such  that for each $t>0$,
	$$\sup_{N} \E{\int_0^t \psi(V^N(s))ds} < +\infty.$$
	Then ${\Gamma_N}$ is relatively compact, and if $\Gamma_N$ converges in law to $\Gamma$, then for $f_1,\dots,f_m \in D_{\psi}$, 
	\begin{align*}
	\big(\int_0^. f_1(V_N(s))& \, ds, \dots , \int_0^. f_m(V^N(s)) \, ds \big) \xrightarrow{\cL}\\
	&\big(\int_{\Rp{3}} f_1(v) \, \Gamma(dv \times [0,.]), \dots , \int_{\Rp{3}} f_m(v) \, \Gamma(dv \times [0,.]) \big)
	\end{align*}
	where $D_{\psi}$ denote the collection of continuous functions $f$ satisfying `\be
	\label{testfunction}\sup_{v \in \Rp{3}} \frac{\va{f(v)}}{\psi(v)}<\infty \quad \text{ and } \lim_{k\to \infty} \sup_{v \in \Rp{3}, \p{v}>k} \frac{\va{f(v)}}{\psi(v)} =0.\ee
\end{lemme}

%\begin{theorem}[Theorem 7.1.4 of \cite{EthierKurtz}]
%	Let $(M^N)_N$ be a sequence of $\bbR^d$-valued martingales such that $M^N(0)=0$. Let $((A^N_{i,j}))$ be symmetric $d\times d$ matrix-valued processes such that $A^N_{i,j}$ has sample paths in $\bbD([0,T], \bbR)$ and $A^N(t)-A^N(s)$ is nonnegative definite for $t>s \geq0$. Assume  
%	$$\lim_{N \to \infty} \E{\sup_{t \in[0,T]} \va{A^N_{i,j}(t)-A_{i,j}^N(t^-)}} = 0$$
%	$$\lim_{N \to \infty} \E{\sup_{t \in[0,T]} \va{M^N(t)-M^N(t^-)}^2} = 0$$ and for all $i,j$, $$M^N_i(t)M^N_j(t)-A^N_{i,j}(t)$$ is a local martingale.
%	
%	Suppose that there exists $C=((c_{ij}))$ deterministic and continuous such that for all $t \geq0$, $$ A^N_{i,j} \to c_{ij}(t)$$. Then the sequence $(M^N)$ converges in law in $\bbD([0,T],\bbR^d)$ to $M$, where $M$ is Gaussian with independent increments and $\E{M(t)M(t)^T}=C(t)$.
%	
%	Suppose furthermore, there exists a matrix $\sigma$ symmetric such that $\frac{d}{dt}C(t)= \sigma(t)^2$, then $M$ can be written as $$M(t) = \int_0^t \sigma(s) dW(s),$$ where $W$ is $d$-dimensional standard Brownian motion. 
%\end{theorem}
\bi
{\bf Aknowledgments}: We warmly thank Vincent Bansaye, the oncologist St\'ephane Giraudier and the biologist Evelyne Lauret  for exciting and fruitful discussions which have motivated this work. We also thank Vincent Bansaye for his precious comments on our paper. This work was supported by a grant from Région Île-de-France.
%The second author also thanks....

	 \end{document}